\documentclass[a4paper,abstracton]{scrartcl}
\usepackage[american]{babel}

\usepackage{amsmath,amssymb,amsthm}
\swapnumbers
\theoremstyle{plain}
\newtheorem{theorem}{Theorem}[section]
\newtheorem{lemma}[theorem]{Lemma}

\theoremstyle{definition}

\theoremstyle{remark}
\newtheorem*{remark}{Remark}

\usepackage[neverdecrease]{paralist}

\def\R{\mathbb{R}}

\let\Phi\phi
\def\toOver#1{\xrightarrow{#1}}
\def\toPhi{\toOver{\Phi}}
\DeclareMathOperator{\LCP}{LCP}
\DeclareMathOperator{\PLCP}{P-LCP}

\DeclareMathOperator{\sgn}{sgn}

\def\subcube#1#2{{#1}[{#2}]}

\begin{document}

\date{July 3, 2013}

\title{Counting Unique-Sink Orientations}

\author{%
Jan Foniok%
	\thanks{Institute for Operations Research, ETH Zurich, 8092 Zurich, Switzerland}
	\thanks{Current address: Department of Mathematics and Statistics, Queen's University,
	Jeffery Hall, 48~University Avenue, Kingston ON K7L 3N6, Canada;
	\texttt{foniok@mast.queensu.ca}}
\and
Bernd G{\"a}rtner%
	\thanks{Institute of Theoretical Computer Science, ETH Zurich, 8092 Zurich,
	Switzerland}
	\thanks{\texttt{gaertner@inf.ethz.ch}}
\and
Lorenz Klaus%
	\footnotemark[1]
	\thanks{\texttt{lklaus@ifor.math.ethz.ch}}
\and
Markus Sprecher%
	\footnotemark[3]
	\thanks{\texttt{markussp@student.ethz.ch}}
}

\maketitle

\begin{abstract}
  Unique-sink orientations (USOs) are an abstract class of orientations
  of the $n$-cube graph. We consider some classes of USOs that are
  of interest in connection with the linear complementarity problem.
  We summarize old and show new lower and upper bounds on the sizes
  of some such classes. Furthermore, we provide a characterization
  of K-matrices in terms of their corresponding USOs.

  \bigskip



\end{abstract}


\section{Introduction}

Unique-sink orientations (USOs) are an abstract class of orientations
of the $n$-cube graph. A number of concrete geometric optimization
problems can be shown to have the combinatorial structure of a USO.
Examples are the linear programming problem~\cite{GarSch:Linear}, and
the problem of finding the smallest enclosing ball of a set of
points~\cite{GarSch:Linear,SzaWel:UniSink}, or a set of
balls~\cite{FisGar:The-smallest}. In this paper, we count the USOs of
the $n$-cube that are generated by P-matrix linear complementarity
problems (P-USOs). This class covers many of the ``geometric'' USOs.
We show that the number of P-USOs is $2^{\Theta(n^3)}$.  The lower
bound construction is the interesting contribution here, and it even
yields USOs from the subclass of K-USOs, whose combinatorial structure
is known to be very rigid~\cite{FonFukGar:Pivoting}. In contrast, the
number of all $n$-cube USOs is doubly exponential
in~$n$~\cite{Mat:The-Number}.

\subsection*{Unique-sink orientations}

We follow the notation of~\cite{FonFukGar:Pivoting}.  Let
$[n]:=\{1,2,\ldots,n\}$.  For a bit vector $v\in\{0,1\}^n$ and
$I\subseteq [n]$, let $v\oplus I$ be the element of~$\{0,1\}^n$
defined by
\[(v\oplus I)_j :=
	\begin{cases}
	1-v_j&\text{if $j\in I$,}\\
	v_j&\text{if $j\notin I$.}
	\end{cases}
\]
Instead of $v\oplus\{i\}$ we write $v\oplus i$.

Under this notation, the (undirected) \emph{$n$-cube} is the graph
$G=(V,E)$ with
\[
V := \{0,1\}^n, \quad
E := \{\{v, v\oplus i\} : v\in V,\ i\in[n]\}.
\] 

A \emph{subcube} of~$G$ is a subgraph $G'=(V',E')$ of~$G$ where
$V'=\{v\oplus I:I\subseteq C\}$ for some vertex $v$ and some set
$C\subseteq[n]$, and $E'=E\cap\binom{V'}{2}$.
The \emph{dimension} of such a subcube is~$|C|$.

Let $\Phi$ be an orientation of the $n$-cube (a digraph with underlying
undirected graph $G$). If $\Phi$~contains the directed edge $(v,
v\oplus i)$, we write $v\toPhi v\oplus i$, or simply $v\to v\oplus i$ if
$\Phi$~is clear from the context. If $V'$ is the vertex set of a
subcube, then the directed subgraph of~$\phi$ induced by~$V'$ is denoted
by~$\subcube{\phi}{V'}$. For $F\subseteq[n]$, let $\Phi^{(F)}$
be the orientation of the $n$-cube obtained by reversing all edges in
coordinates contained in~$F$; formally
\[v \toOver{\Phi^{(F)}} v\oplus i \quad :\Leftrightarrow \quad
	\begin{cases}
	v \toPhi v\oplus i &\text{if $i\notin F$},\\
	v\oplus i \toPhi v &\text{if $i\in F$}.
	\end{cases}
\]

An orientation~$\phi$ of the $n$-cube is a \emph{unique-sink
  orientation} (\emph{USO}) if every subcube $G'=(V',E')$ has a unique
sink (that is, vertex of outdegree zero) in~$\subcube{\phi}{V'}$. It
is not difficult to show that in a unique-sink orientation, every
subcube also has a unique source (that is, vertex of indegree zero).
More generally, if $\phi$~is a unique-sink orientation and $F\subseteq[n]$,
then $\phi^{(F)}$ is a unique-sink orientation as well~\cite[Lemma 2.1]{SzaWel:UniSink}.

A special USO is the \emph{uniform orientation}, in which $v\to v\oplus i$
if and only if $v_i=0$.

Unique-sink orientations enable a graph-theoretic description
of simple principal pivoting algorithms for linear
complementarity problems. They were introduced by Stickney and
Watson~\cite{StiWat:Digraph-models} and have recently received much
attention~\cite{GarMorRus:Unique,GarSch:Linear,GarSpr:Hessenberg,Mat:The-Number,Mor:Randomized-pivot,SchSza:Finding,Sch:Unique,SzaWel:UniSink}.

\subsection*{Linear complementarity problems}

A \emph{linear complementarity problem} ($\LCP(M,q)$) is for a given matrix
$M\in\R^{n\times n}$ and a vector $q\in\R^n$, to find vectors $w,z\in\R^n$
such that
\begin{equation}\label{eq:lcpdef}
w-Mz = q, \quad w,z\geq 0, \quad w^Tz = 0.
\end{equation}

A \emph{P-matrix} is a square real matrix whose principal minors are all
positive. If $M$~is a P-matrix, the appertaining LCP is called a
\emph{P-LCP}; in this case
there exists a unique solution for any~$q$~\cite{STW}.

Let $B\subseteq[n]$, and let $A_B$ be the $n\times n$ matrix whose $j$th
column is the $j$th column of~$-M$ if $j\in B$, and the $j$th column of
the $n\times n$ identity matrix~$I_n$ if $j\notin B$.  If $M$ is a P-matrix,
then $A_B$ is invertible for every set $B$. We call $B$ a \emph{basis}.
If $A_B^{-1}q\geq 0$, let
\begin{equation}\label{eq:lcpsol}
w_i := \begin{cases}
0 & \text{if $i\in B$} \\
(A_B^{-1}q)_i & \text{if $i\notin B$}
\end{cases}, \qquad 
 z_i := \begin{cases}
(A_B^{-1}q)_i & \text{if $i\in B$}\\
0 & \text{if $i\notin B$}
\end{cases}.
\end{equation}
The vectors $w$, $z$ are then a solution to the LCP~\eqref{eq:lcpdef}.

A problem $\PLCP(M,q)$ is \emph{nondegenerate} if $(A_B^{-1}q)_i\neq 0$
for all $B$ and $i$.
Following~\cite{StiWat:Digraph-models}, a nondegenerate $\PLCP(M,q)$
induces a USO:
For $v\in\{0,1\}^n$, let $B(v):=\{j\in[n] : v_j=1\}$.
Then the unique-sink orientation $\Phi$ induced
by $\PLCP(M,q)$ is given by
  \begin{equation}
    \label{eq:usodef}
  v \toPhi v \oplus i  \quad :\Leftrightarrow \quad
  (A_{B(v)}^{-1}q)_i < 0.
  \end{equation}
The run of a simple principal pivoting method
(see~\cite[Chapter~4]{Mur:Linear}) for the P-LCP then corresponds to
following a directed path in the orientation~$\phi$. Finding the sink of the 
orientation is equivalent to finding a basis~$B$ with $A_B^{-1}q\geq 0$, and 
thus via~\eqref{eq:lcpsol} to finding the solution to the P-LCP.

In this paper, we are primarily interested in establishing bounds for the
number of $n$-dimensional USOs satisfying some additional properties (for
instance, USOs induced by P-LCPs), which we introduce in the next section.

\section{Matrix classes and USO classes}
\label{sec:classes}
It is NP-complete to decide whether a solution to an
LCP exists~\cite{Chung:Hardness}. If the matrix~$M$ is a
P-matrix, however, a solution always exists. The problem of
finding it is unlikely to be NP-hard, because if it were, then
NP${}={}$co-NP~\cite{Meg:A-Note-on-the-Complexity}.  Even so, no
polynomial-time algorithms for solving P-LCPs are known. Hence our
motivation to study some special matrix classes and investigate what
combinatorial properties their USOs have. The ultimate goal is then
to try and exploit these combinatorial properties in order to find an
efficient algorithm for the corresponding LCPs.

A \emph{Z-matrix} is a square matrix whose off-diagonal entries
are all non-positive. A \emph{K-matrix} is a matrix which is both a Z-matrix
and a P-matrix. A \emph{hidden-K-matrix} is a P-matrix~$M$ such that there
exist Z-matrices $X$ and~$Y$ and non-negative vectors $r$ and~$s$ with
$MX=Y$, $r^TX+s^TY>0$. Taking $X$ to be the identity matrix and $Y=M$, $s=0$
and $r$ any positive vector
shows that every K-matrix is a hidden-K-matrix as well.

The importance of these matrix classes is due to the fact that
polynomial-time algorithms are known for solving the $\LCP(M,q)$
if the matrix~$M$ is a Z-matrix~\cite{Cha:A-special,Sai:A-note}, a
hidden-K-matrix~\cite{Man:Linear}, or the transpose of a
hidden-K-matrix~\cite{PanCha:Linear}.

A USO is a \emph{P-USO} if it is induced via~\eqref{eq:usodef} by some
$\LCP(M,q)$ with a P-matrix~$M$; it is a \emph{K-USO} if it is induced
by some
$\LCP(M,q)$ with a K-matrix~$M$; and it is a \emph{hidden-K-USO} if it
is induced by some $\LCP(M,q)$ with a hidden-K-matrix~$M$.

A USO is a \emph{Holt--Klee USO} if in each of its subcubes, there are
$d$~directed paths from the source to the sink of the subcube, with no
two paths sharing a vertex other than source and sink; here $d$~is the
dimension of the subcube. A USO~$\phi$ is \emph{strongly Holt--Klee}
if $\phi^{(F)}$ is Holt--Klee for every $F\subseteq[n]$.
By~\cite{GarMorRus:Unique}, every P-USO is a strongly Holt--Klee USO.

Finally, a USO is \emph{locally uniform}, if
\begin{multline}
\label{eq:up-uni}
\text{whenever } v_i=v_j=0 \text{ and } v\toPhi v\oplus i,\ v\toPhi v\oplus j,\\
\text{then } v\oplus i \toPhi v\oplus\{i,j\},\ v\oplus j\toPhi v\oplus\{i,j\}
\end{multline}
and
\begin{multline}
\label{eq:down-uni}
\text{whenever } v_i=v_j=0 \text{ and } v\oplus i\toPhi v,\ v\oplus j \toPhi v,\\
\text{then } v\oplus\{i,j\}\toPhi v\oplus i,\ v\oplus\{i,j\}\toPhi v\oplus j.
\end{multline}
By~\cite{FonFukGar:Pivoting}, every K-USO is locally uniform, and every locally
uniform USO is acyclic. We thus have the following chain of
inclusions, some of which are in fact strict:
\begin{multline*}
\text{K-USOs}\subseteq\text{locally uniform P-USOs}
\subset\text{acyclic P-USOs}\subset\text{P-USOs}\\
\subset \text{strongly Holt--Klee USOs}\subset\text{Holt--Klee USOs}.
\end{multline*}
The first inclusion is not known to be strict; see also
Section~\ref{sec:loc}. The second inclusion is easily seen to be strict
already for $n=2$. Strictness of the third inclusion is due to 
Stickney and Watson: there exists a cyclic P-USO of the 
3-cube~\cite{StiWat:Digraph-models}. The fourth inclusion is strict
as a consequence of our counting results: the number of strongly
Holt--Klee USOs is much larger than the number of P-USOs. Finally,
there is an example that shows strictness of the fifth inclusion~\cite[Fig.~12]{GarMorRus:Unique}.

An \emph{LP-USO} is an orientation of the $n$-cube admitting a realization
$r:\{0,1\}^n\to\R^n$
as a polytope in the $n$-dimensional Euclidean space, combinatorially
equivalent to the $n$-cube, such that there
exists a linear function~$f:\R^n\to\R$ and
\[v\toPhi v\oplus i \quad\text{if and only if}\quad f(r(v\oplus i)) > f(r(v)).\]
It follows
from~\cite{Kla:On-Classes,Kla:Fresh,Mor:Distinguishing,MorLaw:Geometric,PanCha:Linear}
that LP-USOs are exactly hidden-K-USOs, and we have:
\[\text{K-USOs}\subset\text{LP-USOs}=\text{hidden-K-USOs}\subseteq
\text{acyclic P-USOs}.\]
Again, the first inclusion is strict already for $n=2$; it is open
whether the last inclusion is strict. 
In the next section we examine the numbers of $n$-USOs in the respective classes.

It is also possible to obtain USOs from completely general linear
programs. The reduction in~\cite{GarSch:Linear} yields \emph{PD-USOs},
i.e., USOs generated by LCPs with symmetric positive definite matrices
$M$. Since these are exactly the symmetric P-matrices~\cite[Section
3.3]{CotPanSto:LCP}, we also have
$\text{PD-USOs}\subseteq\text{P-USOs}$, where we do not know whether
the inclusion is strict.
The USOs that are obtained from the problem of finding the smallest
enclosing ball of a set of points~\cite[Section 3.2]{Gar:RandAlgs} are
``almost'' PD-USOs in the sense that every subcube not containing the
origin $0$ is oriented by a PD-USO~\cite{hiro}. For the USOs from
smallest enclosing balls of \emph{balls}~\cite{FisGar:The-smallest}, 
we are not aware of a similar result. 

\section{Counting USOs}

In this section, we examine the number of USOs in the classes
described above, depending on their dimension. The $n$-cube as we have
introduced it is a labelled graph; accordingly, the counting will be
in the labelled sense. But all the bounds are valid also for the
number of isomorphism classes of USOs: the $n$-cube has $2^n
n!=2^{\Theta(n\log n)}$ automorphisms, so the labelled and unlabelled
counts differ by at most this factor---which is negligible, since all
our bounds are at least of the order $2^{\Omega(n^3)}$.

First counting results about USOs were obtained by Matou\v
sek~\cite{Mat:The-Number}, who gave asymptotic bounds on the number of
all USOs and acyclic USOs.

Next, Develin~\cite{Dev:LP-orientations}---in order to show that
the Holt--Klee condition does not characterize LP-USOs---proved
that the number of $n$-dimensional LP-USOs is bounded from above by
$2^{O(n^3)}$, whereas the number of Holt--Klee USOs is bounded
from below by~$2^{\Omega(2^n/\sqrt{n})}$.

Using similar means, we prove an upper bound of~$2^{O(n^3)}$
on the number of P-USOs, and observe that a slight
modification of Develin's construction yields a lower bound
of~$2^{\binom{n-1}{\lfloor(n-1)/2\rfloor}}=2^{\Omega(2^n/\sqrt{n})}$ for strongly Holt--Klee
locally uniform USOs. Furthermore, we provide a construction of
$2^{\Omega(n^3)}$ K-USOs.
These results imply that the number of K-USOs, LP-USOs, as well
as P-USOs, is $2^{\Theta(n^3)}$.

Previously known and new bounds on the number of $n$-dimensional
USOs in the classes defined in the previous section are summarized in
the following table. Where an entry is missing, the best known bound
coincides with the one of a subclass or a superclass; see also
Section~\ref{sec:classes}. We note that already before Develin's counting 
result~\cite{Dev:LP-orientations}, it had been shown by Morris~\cite{Mor:Distinguishing} that the Holt--Klee
condition does not characterize P-USOs, starting from dimension $n=4$.

\newdimen\wwiddthb
\newdimen\wwiddthc
\setbox0\hbox{$2^{\Omega(2^n/\sqrt{n})}$~\cite{Dev:LP-orientations}}
\wwiddthb=\wd0
\setbox0\hbox{$2^{O(n^3)}$~\cite{Dev:LP-orientations}}
\wwiddthc=\wd0
\def\wdb#1{\hbox to \wwiddthb{#1\hfil}}
\def\wdc#1{\hbox to \wwiddthc{#1\hfil}}

\begin{center}
\begin{tabular}{l|cc}
class&lower bound&upper bound\\
\hline
K-USOs&					\wdb{$2^{\Omega(n^3)}$}\\
LP-USOs&			&					\wdc{$2^{O(n^3)}$~\cite{Dev:LP-orientations}}\\
P-USOs&				&					\wdc{$2^{O(n^3)}$}\\
acyclic strongly Holt--Klee USOs&	\wdb{$2^{\Omega(2^n/\sqrt{n})}$}\\
Holt--Klee USOs&			\wdb{$2^{\Omega(2^n/\sqrt{n})}$~\cite{Dev:LP-orientations}}\\
locally uniform USOs&			\wdb{$2^{\Omega(2^n/\sqrt{n})}$}\\
acyclic USOs~\cite{Mat:The-Number}&	\wdb{$2^{2^{n-1}}$}&		\wdc{$(n+1)^{2^n}$}\\
all USOs~\cite{Mat:The-Number}&		\wdb{$n^{\Omega(2^n)}$}&	\wdc{$n^{O(2^n)}$}
\end{tabular}
\end{center}

\subsection{An upper bound for P-USOs}\label{sec:upperboundP}

Every P-USO is determined by the sequence $\sigma(M,q)=\bigl(\sgn
(A_{B(v)}^{-1}q)_i : v\in \{0,1\}^n,\ i\in[n]\bigr)$, which is a function
of the P-matrix~$M$ and the right-hand side~$q$. Furthermore, we are
interested only in nondegenerate right-hand sides~$q$, which means we are
interested only in sequences containing no~$0$.

\begin{lemma}
\label{lem:uso-poly}
Each entry of the vector $\sigma(M,q)$ is the sign of a polynomial in
the entries of~$M$ and~$q$ of degree at most~$n$.
\end{lemma}

\begin{proof}
The entries of the matrix $A_{B(v)}^{-1}$ can be computed as
\[
(A_{B(v)}^{-1})_{rs} = \frac{1}{\det A_{B(v)}} (-1)^{r+s} A_{sr},
\]
where $A_{ij}$ is the determinant of the submatrix of~$A_{B(v)}$ obtained
by deleting the $i$th row and the $j$th column, which is a polynomial
of degree at most~$n-1$. Hence
\[
(A_{B(v)}^{-1}q)_i = 
	\frac{1}{\det A_{B(v)}} \sum_{s=1}^n q_s \cdot (-1)^{i+s} \cdot A_{si}.
\]

Recall that $A_{B(v)}$ has $|B(v)|$ columns of~$-M$ and $n-|B(v)|$
columns of the identity matrix; thus $\sgn\det A_{B(v)} = (-1)^{|B(v)|}$,
since $M$~is a P-matrix. Therefore
\[
\sgn (A_{B(v)}^{-1}q)_i =
	\sgn\Bigl( (-1)^{|B(v)|} \cdot
	\sum_{s=1}^n q_s \cdot (-1)^{i+s} \cdot A_{si} \Bigr),
\]
which is the sign of a polynomial of degree at most~$n$.
\end{proof}

The algebraic tool we will apply is the following theorem.

\begin{theorem}[Warren \cite{Warren}]
\label{thm:warren}
Let $p_1,\dotsc,p_\ell$ be real polynomials in $k$~variables, each of
degree at most~$d$. For $\ell\ge k$, the number of sign sequences
$\sigma(x)=(\sgn p_1(x),\dotsc, \sgn p_\ell(x))$ that consist of terms
$+1$,~$-1$ is at most~$(4ed\ell/k)^k$.
\end{theorem}

Now all is set to prove an upper bound on the number of P-USOs.

\begin{theorem}
\label{thm:p-upper}
The number of distinct $n$-dimensional P-USOs is at most~$2^{O(n^3)}$.
\end{theorem}

\begin{proof}
By Lemma~\ref{lem:uso-poly}, each P-USO is determined by a vector
of $\ell=n2^n$ nonzero signs of polynomials of degree at most~$n$. The number of
variables is $k=n^2+n$ (equal to the number of entries of the matrix~$M$
and the vector~$q$). By Theorem~\ref{thm:warren}, there are at most
\[
\left(\frac{4e \cdot n \cdot n2^n}{n^2+n}\right)^{n^2+n} 
\le \left(4e\cdot 2^n\right)^{n^2+n} = 2^{O(n^3)}
\]
such sign vectors.
\end{proof}

\subsection{A lower bound for strongly Holt--Klee and locally uniform USOs}

Recall that a \emph{strongly Holt--Klee orientation} is a Holt--Klee
orientation~$\phi$ that remains so after flipping all the edges in
any given subset of coordinates, that is, if $\phi^{(F)}$~is
Holt--Klee for any $F\subseteq[n]$.

A \emph{monotone Boolean function} is a function $f:\{0,1\}^k\to\{0,1\}$
such that if $x\le y$, then $f(x)\le f(y)$; in~``$x\leq y$'',
$\leq$~is to be understood component-wise. Counting monotone Boolean
functions is known as \emph{Dedekind's
problem}~\cite{Ded:Ueber-Zerlegungen}. Let $M$ be the set of
$0,1$-vectors of length~$k$ with exactly $\lfloor k/2\rfloor$ ones.
Following~\cite{Kle:On-Dedekinds-Problem:}, a lower bound of
$2^{\binom{k}{\lfloor k/2\rfloor}}$ on the number of $k$-variate
monotone Boolean functions can be obtained by
taking for each subset $A\subseteq M$ the function~$f_A$ given by
\[f_A(x)=1 \quad\text{iff}\quad \{y\in A:y\le x\}\ne\emptyset.\]
This means, $f_A$~attains value~$1$ exactly on the $0,1$-vectors in~$A$
and all the ones that are larger (w.r.t. the order~$\leq$).

\begin{theorem}\label{thm:lu}
The number of acyclic locally uniform strongly Holt--Klee $n$-USOs is at
least $2^{\binom{n-1}{\lfloor (n-1)/2\rfloor}}=2^{\Omega(2^n/\sqrt{n})}$.
\end{theorem}

\begin{proof}
Given an $(n-1)$-variate monotone Boolean function~$f$, we construct an
$n$-USO~$\phi$ by setting
\begin{align*}
v\toPhi v\oplus i & \text{ if $i\ne n$ and $v_i=0$},\\
v\toPhi v\oplus n & \text{ if $v_n+f(v')=1$,}
\end{align*}
where $v'\in\{0,1\}^{n-1}$ is formed by the initial $n-1$ bits of~$v$,
and addition in the second equation is modulo~$2$. It is easy to see
that this indeed defines a USO: on both facets $\{v: v_n=0\}$ and
$\{v: v_n=1\}$, we have the same (uniform) orientation, and this is
already sufficient to guarantee the USO properties. Between the two
facets, we have $(v',0)\toPhi(v',1)$ if and only if $f(v')=1$.

The USO~$\phi$ is clearly acyclic because any directed walk
in~$\phi$ is monotone on the first $n-1$ bits. It is easy to show local
uniformity too. The assumption of~\eqref{eq:down-uni} is never satisfied.
For~\eqref{eq:up-uni} it suffices to consider the case $j=n$: If $v\toPhi
v\oplus n$, then $f(v')=1$, thus by monotonicity $f((v\oplus i)')=1$.
Hence $v\oplus i\toPhi v\oplus\{i,n\}$.

For the strong Holt--Klee property, let $F\subseteq[n]$ and let
$V'=\{v\oplus I:I\subseteq C\}$ be the vertex set of a subcube with $|C|=:d$. If
$n\notin C$, then $\phi^{(F)}[V']$~is isomorphic to the uniform
orientation, which is easily seen to satisfy the Holt--Klee property. So
suppose $n\in C$. Let $V_0:=\{v\in V':v_n=0\}$ and $V_1:=\{v\in V':v_n=1\}$
and let $s$ be the source and $t$ the sink of~$\phi^{(F)}[V']$. Note that
$\phi^{(F)}[V_0]$ and $\phi^{(F)}[V_1]$ are identical if we truncate the
last coordinate of their vertices, and isomorphic to the uniform USO.

Now we distinguish two cases. First, if both $s$ and~$t$ lie in the
same set $V_0$ or~$V_1$, that is, if $b:=s_n=t_n$, then there are $d-1$
disjoint paths from~$s$ to~$t$ in~$\phi^{(F)}[V_b]$ and another
path obtained by concatenating the edge $s\to s\oplus n$, a path in
$\phi^{(F)}[V_{1-b}]$ from~$s\oplus n$ to~$t\oplus n$, and the edge
$t\oplus n\to t$.

Second, let $b:=s_n=1-t_n$. Without loss of generality we may assume
that $b=0$ and $n\notin F$. Let $P(i_1,\dotsc,i_d)$ denote the
directed path $s\to s\oplus\{i_1\}\to s\oplus\{i_1,i_2\}\to\dotsb\to
s\oplus\{i_1,i_2, \dotsc,i_d\}$. Order the elements of
$C\setminus\{n\}=\{j_1,j_2,\dotsc,j_{d-1}\}$ so that for $j_k\in
F$ and $j_\ell\notin F$ we have $k<\ell$. Since $\phi^{(F)}[V_0]$
and $\phi^{(F)}[V_1]$ are both isomorphic to the uniform orientation
and $s_n\ne t_n$, we have $t=s\oplus C$. Now we claim that the paths
$P(j_1,j_2,\dotsc,j_{d-1},n)$, $P(j_2,j_3,\dotsc,j_{d-1},n,j_1)$,~\dots,
$P(j_{d-1},n,j_1,j_2,\dotsc,j_{d-2})$, $P(n,j_1,j_2,\dotsc,j_{d-1})$
are vertex-disjoint directed paths from~$s$ to~$t$. The only non-obvious
fact to show is that for any~$k\in[d]$, there is a directed edge
$u:=s\oplus\{j_k,j_{k+1},\dotsc,j_{d-1}\}\to
v:=s\oplus\{j_k,j_{k+1},\dotsc,j_{d-1},n\}$.
Note that $u\to v$ if and only if $f(u')=1$ and that $f(s')=f(t')=1$.
If $j_k\notin F$, then $s'\leq u'$ and so $1=f(s')\le f(u')$, thus
$f(u')=1$.
If on the other hand $j_k\in F$, then $t'\le u'$ and so $1=f(t')\le f(u')$,
thus $f(u')=1$. Hence $u\to v$.

Therefore the number of acyclic locally uniform strongly Holt--Klee
$n$-USOs is lower bounded by the number of $(n-1)$-variate monotone
Boolean functions, which concludes the proof.
\end{proof}

\begin{remark}
After swapping the roles of $0$ and $1$ in the $n$th coordinate, the
above construction is the same as Mike Develin's
construction~\cite{Dev:LP-orientations} of many orientations
satisfying the Holt--Klee condition.  Thus,
both Develin's and our construction yield Holt--Klee orientations, but
local uniformity is obtained only in our variant.

The logarithm of the total number of acyclic $n$-USOs is
no more than $2^n\log(n+1)$ \cite{Mat:The-Number}. In comparison,
the exponent in the lower bound obtained from Theorem~\ref{thm:lu} is
of the order $2^n/\sqrt{n}$, and therefore still exponential. 
Restricting to K-USOs, the exponent goes down to a polynomial in~$n$.
\end{remark}


\subsection{A lower bound for K-USOs}\label{sec:KUSO}

\begin{theorem}
The number of distinct K-USOs in dimension~$n$ is at
least~$2^{\Omega(n^3)}$.
\end{theorem}

\begin{proof}
Consider the upper triangular matrix
\[
M(\beta) =
\begin{pmatrix}
1& -1-\beta_{1,2}& -1-\beta_{1,3}& \hdots & -1-\beta_{1,n}\\
0& 1& -1-\beta_{2,3} & \hdots & -1-\beta_{2,n}\\
\hdotsfor5\\
0& 0& 0& \hdots & -1-\beta_{n-1,n}\\
0& 0 & 0 & \hdots & 1
\end{pmatrix}
\]
and the vector $q=(-1,1,-1,\dotsc,(-1)^n)^T$. If the
parameters~$\beta_{i,j}$ are sufficiently small in absolute value,
then $M(\beta)$ is a K-matrix. We will now examine how the
choice of the~$\beta_{i,j}$ influences the USO induced by the
$\LCP(M(\beta),q)$. Our goal is to show that we can make $2^{\Omega(n^3)}$
choices, each of which induces a different USO.

From now on, we will write
\[
(i,j)\prec (i',j')
	\text{\quad for\quad} (j<j') \text{ or } (j=j' \text{ and } i>i').
\]
Note that $\prec$ is a total ordering on $\{(i,j)\in[n]^2:i<j\}$.
The strategy will be to choose the values of the $\beta_{i,j}$ in
the order given by~$\prec$, that is, from left to right and in each
column from bottom to top. We show that for about half of
the~$\beta_{i,j}$'s there is a number of choices exponential in
$j-i$ such that each of these choices determines a different
orientation on a certain subset of edges, which will be independent
of all subsequently made choices. But first we examine the expressions
that determine the orientation.

Let $B\subseteq[n]$ and, analogously to the definition of~$A_B$,
let $A_B(\beta)$ be the matrix whose $j$th column is the $j$th column
of~$-M(\beta)$ if $j\in B$, and the $j$th column of~$I_n$ otherwise.

\begin{lemma}\label{lem:invAB}
The entries of the inverse matrix $(A_B(\beta))^{-1}$ of $A_B(\beta)$ satisfy
\begin{equation}\label{eq:ainv}
\sigma_r\cdot\bigl((A_B(\beta))^{-1}\bigr)_{r,s} =
\begin{cases}
1& \text{if $r=s$,}\\
0& \text{if $r>s$ or if $r<s$ and $s\notin B$,}\\
2^{p(B,r,s)} + \beta_{r,s} + t_{B,r,s}(\beta)& \text{if $r<s$ and $s\in B$,}
\end{cases}
\end{equation}
where $\sigma_r = -1$ if $r\in B$ and $\sigma_r=1$ if $r\notin B$,
$p(B,r,s)=|\{j\in B: r<j<s\}|$ and $t_{B,r,s}(\beta)$ is a polynomial
with positive coefficients and no constant term, in exactly the
variables $\beta_{i,j}$ for 
\begin{equation}
\label{eq:Jrs}
(i,j)\in J_{r,s}(B) := \{(i,j)\in (B\cup\{r\})\times B: r\leq i<j, (i,j)\prec(r,s)\}.
\end{equation}
\end{lemma}

\begin{proof}
  The inverse of an upper triangular matrix is again upper triangular.
  Its diagonal entries are the reciprocals of the diagonal entries of
  the original matrix; in our case, they are~$\pm1$. Moreover, for
  $s\notin B$, the $s$th column of $(A_B(\beta))^{-1}$ equals the
  $s$th unit vector $e_s$, the unique solution of the equation system
  $A_B(\beta)x=e_s$.

So it remains to examine the above-diagonal entries in columns
belonging to~$B$. Such entries only exist if $B\setminus\{1\}\neq\emptyset$,
and they are indexed by
$J(B)=\{(r,s)\in[n]\times B\colon r<s\}$. Consider the ordering~$\prec$
defined above, restricted to~$J(B)$. The least element of~$J(B)$
with respect to~$\prec$ is $(s-1,s)$, where $s$~is the least element
of $B\setminus\{1\}$. Multiplying row ${s-1}$ of~$A_B(\beta)$ with
the $s$th column of $(A_B(\beta))^{-1}$ reveals that
$\left((A_B(\beta))^{-1}\right)_{s-1,s}=\sigma_{s-1}\cdot(1+\beta_{s-1,s})$.
Thus \eqref{eq:ainv}~holds for the $\prec$-minimum $(r,s)$ in~$J(B)$.

For any other $(r,s)\in J(B)$, assume that \eqref{eq:ainv}~holds
for all $(k,s)\prec(r,s)$. Multiplying the $r$th row of~$A_B(\beta)$
by the $s$th column of~$(A_B(\beta))^{-1}$ shows that
\begin{eqnarray*}
\underbrace{\sigma_r}_{A_B(\beta){r,r}}\cdot\left( (A_B(\beta))^{-1}\right)_{r,s} 
&+& \sum_{\substack{k\in B\\ r<k<s}}
\underbrace{(1+\beta_{r,k})}_{A_B(\beta)_{r,k}}\cdot
\underbrace{\sigma_k \cdot
	\left( 2^{p(B,k,s)} + \beta_{k,s} + t_{B,k,s}(\beta) \right)}_{((A_B(\beta))^{-1})_{k,s}} \\
&+& \underbrace{(1+\beta_{r,s})}_{A_B(\beta)_{r,s}} \cdot 
\underbrace{\sigma_s}_{((A_B(\beta))^{-1})_{s,s}} =0.
\end{eqnarray*}
As $\sigma_k=\sigma_s=-1$ for $k,s\in B$, we have
\begin{equation}
\label{eq:tBrs}
\begin{split}
\sigma_r\cdot\left((A_B(\beta))^{-1}\right)_{r,s} &=
\beta_{r,s} + 1 + \sum_{\substack{k\in B\\r<k<s}}(1+\beta_{r,k})
	\left(2^{p(B,k,s)}+\beta_{k,s}+t_{B,k,s}(\beta)\right)\\
&= \beta_{r,s} + 1 + \sum_{\substack{k\in B\\r<k<s}} 2^{p(B,k,s)}
	+ t_{B,r,s}(\beta) \\
&= \beta_{r,s} + 2^{p(B,r,s)} + t_{B,r,s}(\beta),
\end{split}
\end{equation}
where 
\begin{equation}
\label{eq:tBrs-rec}
t_{B,r,s}(\beta) := \sum_{\substack{k\in B\\r<k<s}} \left(
2^{p(B,k,s)}\beta_{r,k} + \beta_{k,s} + \beta_{r,k}\beta_{k,s}
+ t_{B,k,s}(\beta) + \beta_{r,k} t_{B,k,s}(\beta)
\right )
\end{equation}
is a polynomial with positive coefficients and no constant term, as
required. The variables appearing in this polynomial are indexed by 
the set
\[
\bigcup_{\substack{k\in B\\r<k<s}}\bigl(\{(r,k),(k,s)\}\cup
  J_{k,s}(B)\bigr) = J_{r,s}(B).
\qedhere
\]
\end{proof}

We next investigate the vectors $(A_B(\beta))^{-1}q$ whose sign
patterns determine the orientations of the edges in the unique sink
orientation; see~(\ref{eq:usodef}). First we identify a large number of
bases $B$ for which the signs in $(A_B(\beta))^{-1}q$ are sensitive to
very small changes in $\beta$.

Let $B\subseteq[n]$ be a basis such that, for $m=\max B$, we have $s\equiv
m+1\pmod{2}$ for each $s\in B\setminus\{m\}$. Then $q_m\cdot q_s=-1$ for
each $s\in B\setminus\{m\}$, and hence for all $r<m$ such that $r\equiv m+1\pmod2$,
\begin{align}
\sigma_r \cdot \left( (A_B(\beta))^{-1}q\right)_r &=
(-1)^m
\biggl(\beta_{r,m}+ t_{B,r,m}(\beta) -
	\sum_{\substack{s\in B\\r<s<m}} \bigl(\beta_{r,s}+t_{B,r,s}(\beta)\bigr)\biggr)
\notag\\
&=
(-1)^m
\bigl( \beta_{r,m} - t'_{B,r,m}(\beta)\bigr);
\label{eq:tdash}
\end{align}
the parity condition on $r$ and $m$
ensures that the constant terms sum to zero. Each $t'_{B,r,m}(\beta)$
is some polynomial in variables $\beta_{i,j}$ for $(i,j)\in\{(i,j)\in
[n]\times B: i<j,\ (i,j)\prec(r,m)\}$ with no constant term.
In particular, this implies that in the corresponding USO, the
orientation of the $r$th edge incident to the vertex corresponding
to the basis~$B$ depends only on the values of~$\beta_{i,j}$ with
$(i,j)\preceq(r,m)$.

Now let $r,m\in[n]$, $r<m$, $r\equiv m+1\pmod2$. Let
\[ C = C(r,m) = \{i\in[n]: r<i<m,\ i\equiv m+1\!\!\!\!\!\pmod{2}\}\]
and let
\[ V' = V'(r,m) = \{ (0\oplus m) \oplus I : I\subseteq C\}.\]
Note that $|C|=(m-r-1)/2$ and so $|V'|=2^{(m-r-1)/2}$.

Furthermore, suppose for a moment that the values of $\beta_{i,j}$ are fixed for
all $(i,j)\prec(r,m)$, and that these values satisfy:
\begin{equation}
\label{eq:generic}
v,v'\in V',\ v\ne v' \quad \Longrightarrow\quad 
	t'_{B(v),r,m}(\beta) \ne t'_{B(v'),r,m}(\beta).
\end{equation}
For each $v\in V'$, the direction of the edge between $v$ and $v\oplus r$
in the USO induced by $\LCP(M(\beta),q)$ is by~\eqref{eq:tdash} determined by the sign of the
difference $\beta_{r,m} - t'_{B(v),r,m}(\beta)$. By~\eqref{eq:generic},
the currently fixed values of $t'_{B(v),r,m}$ for $v\in V'$ are all 
distinct and thus they split the reals into $|V'|+1$ intervals. Hence
there are $|V'|+1$ choices for $\beta_{r,m}$ so that the resulting
USOs will differ from one another in the orientation of at least one of
these edges.

What happens, though, if we are about to choose~$\beta_{r,m}$ and
\eqref{eq:generic}~is not satisfied?  Then we have to revise the
choices we have made so far. Slightly perturbing each~$\beta_{i,j}$
with $(i,j)\prec(r,m)$ will not change the orientation (because
each $\beta_{i,j}$ is chosen in the \emph{interior} of one of the
$|V'|+1$ intervals mentioned above); the next
lemma implies that it will make \eqref{eq:generic} satisfied.

\begin{lemma}
\label{lem:tees}
Let $r,m\in[n]$, $r<m$, $r\equiv m+1\pmod 2$ and let $B_1,B_2\subseteq
[n]$ be bases such that $\max B_1 = \max B_2 = m$, $\min B_1>r$, $\min
B_2>r$, and 
that $i\equiv m+1\pmod 2$ for all $i\in(B_1\cup B_2)\setminus\{m\}$. Then the
polynomial $t'_{B_1,r,m}(\beta)-t'_{B_2,r,m}(\beta)$ is identically
zero if and only if $B_1=B_2$.
\end{lemma}

\begin{proof}
First, from \eqref{eq:tdash} we have:
\begin{equation}
\label{eq:tdash-exp}
t'_{B,r,m}(\beta) = -t_{B,r,m}(\beta) + \sum_{\substack{s\in B\\r<s<m}}
	\bigl(\beta_{r,s} + t_{B,r,s}(\beta) \bigr).
\end{equation}
Assume that $B_1\ne B_2$. Without loss of generality, there exists
some $u\in B_1\setminus B_2$; by assumption $u>r$. It follows
from~(\ref{eq:tdash-exp}) and the properties of the polynomials $t$
guaranteed by Lemma~\ref{lem:invAB} that $t'_{B_1,r,m}(\beta)$
contains the variable $\beta_{u,m}$ while $t'_{B_2,r,m}(\beta)$ does not.
Hence $t'_{B,r,m}(\beta)-t'_{B',r,m}(\beta)$ is not identically
zero. The converse implication is trivial.
\end{proof}

The options to choose $\beta_{r,m}$ are, of course, not independent
of the values of the other~$\beta_{i,j}$'s. However, they depend only
on the $\beta_{i,j}$'s with $(i,j)\prec(r,m)$. Hence it is possible to
make the choices sequentially in the order given by~$\prec$; starting
with~$\beta_{1,2}$ and finishing with~$\beta_{1,n}$.
The values of $\beta_{r,m}$ for $r\equiv m\pmod 2$ can be chosen
arbitrarily, e.g., $\beta_{r,m}=0$.

Therefore the number of distinct USOs induced by $\LCP(M(\beta),q)$
for various values of $\beta_{i,j}$, as described above, is at least
\begin{equation}\label{eq:bigprod}
\prod_{m=1}^{n} \prod_{\substack{1\le r<m\\r\equiv m+1\!\!\!\!\!\pmod2}}
	\left(2^{(m-r-1)/2}+1\right) =
\prod_{m=1}^{n} \prod_{i=0}^{\lfloor m/2\rfloor-1}
	\left( 2^{i} +1 \right) =
2^{\Omega(n^3)}.
\end{equation}

Finally, it remains to show that the values of all $\beta_{i,j}$'s
can be chosen to satisfy $|\beta_{i,j}|<1$, so that $M(\beta)$~would
be a K-matrix. That follows from the next lemma.

\begin{lemma}
\label{lem:betas}
Whenever $t'_{B,r,m}$ as in (\ref{eq:tdash-exp}) is defined, let
\[\bar\beta=\max\bigl\{|\beta_{i,j}|: (i,j)\in
[n]\times B,\  i<j,\ (i,j)\prec(r,m)\bigr\}.\]
If $\bar\beta<1$, then $|t'_{B,r,m}(\beta)| < 4^{m-r+1} \bar\beta$.
\end{lemma}

\begin{proof}
By definition, $p(B,j,s)\leq s-j-1$ for all eligible $B,j,s$. Now
we claim that
\begin{equation}
\label{eq:bd1}
|t_{B,r,s}(\beta)| \leq 4^{s-r} \bar\beta ,
\end{equation}
with $t_{B,r,s}(\beta)$ as in (\ref{eq:tBrs-rec}).
If $s-r=1$ or $\bar\beta=0$, then by \eqref{eq:tBrs-rec} we have
$t_{B,r,s}(\beta)=0\le 4^{s-r}\bar\beta$. Otherwise, by induction
on $s-r$, again using~\eqref{eq:tBrs-rec} and $\bar\beta^2<\bar\beta<1$, 
we have
\begin{align*}
|t_{B,r,s}(\beta)|
&\le \sum_{j=r+1}^{s-1} \left(2^{p(B,j,s)}\bar\beta + \bar\beta +
  \bar\beta^2 + 4^{s-j}\bar\beta +  4^{s-j}\bar\beta^2\right)\\
&\le \sum_{j=r+1}^{s-1} \left(2^{s-j-1} + 2 + 2\cdot
  4^{s-j}\right) \bar\beta 
\le \sum_{j=r+1}^{s} \left(3\cdot
  4^{s-j}\right) \bar\beta =(4^{s-r}-1)\bar\beta \leq 4^{s-r}\bar\beta.
\end{align*}
Thus \eqref{eq:bd1}~holds.

Finally, unless $\bar\beta=0$, in which case $t'_{B,r,s}(\beta)=0$,
we conclude from~\eqref{eq:tdash-exp} and \eqref{eq:bd1} that
\begin{align*}
|t'_{B,r,m}(\beta)|
&\leq \biggl(4^{m-r} + (m-r-1) + \sum_{s=r+1}^{m-1} 4^{s-r}\biggr)\bar\beta \\
&=\bigl(m-r-1 + \textstyle\frac13(4^{m-r+1}-4) \bigr) \bar\beta \\
&\leq 4^{m-r+1} \bar\beta.
\qedhere
\end{align*}
\end{proof}

The first $\beta$ to be chosen is $\beta_{1,2}$, and its sign
determines the direction of the edge between $(0,1,0,\dotsc,0)$ and
$(1,1,0,\dotsc,0)$. If $\beta_{1,2}$ is chosen to be $\pm (4+\epsilon)^{-n^3}$,
then all subsequent choices can be made in such a way that
$|\beta_{r,s}|<1$ for all $r,s$.
\end{proof}

\subsection{The number of USOs from a fixed matrix}\label{sec:fixedM}

In this section, we prove the following 

\begin{theorem} For a P-matrix $M\in\R^{n\times n}$, let $u(M)$ be the number
of USOs determined by LCPs of the form $\LCP(M,q)$ for $q\in\R^n$. Furthermore,
define $u(n) = \max_{M} u(M)$, where the maximum is over all $n\times n$
P-matrices. Then
\[
u(n) = 2^{\Theta(n^2)}.
\] 
\end{theorem}

\begin{proof}
  Let us first show the upper bound. For a fixed $M$, we consider the
  $n2^n$ hyperplanes of the form
  \[\bigl\{x\in\R^n: (A_{B(v)}^{-1}x)_i = 0\bigr\}.\] These hyperplanes
  determine an \emph{arrangement} that subdivides $\R^n$ into faces of
  various dimensions. Each face is an inclusion-maximal region over
  which the sign vector $\bigl(\sgn(A_{B(v)}^{-1}x)_i:
  v\in\{0,1\}^n,i\in[n]\bigr)$ is constant. The faces of dimension $n$ are
  called \emph{cells}; within a cell, the sign vector is nonzero
  everywhere.  From Section~\ref{sec:upperboundP} we know that
  $\LCP(M,q)$ yields a USO whenever $q$ is in some cell, and for all
  $q$ within the same cell, $\LCP(M,q)$ yields the same USO. Thus,
  the number of cells in the arrangement is an upper bound for the
  number $u(M)$ of different USOs induced by $M$. It is
  well-known~\cite{e-acg-87} that the number of cells in an
  arrangement of $N$ hyperplanes in dimension $n$ is $O(N^n)$. In our
  case, we have $N=n2^n$ which shows that $u(M)= O((n2^n)^n) =
  2^{O(n^2)}$ for all $M$.

  For the lower bound, note that we have in particular constructed in
  Section~\ref{sec:KUSO} a K-matrix $M'\in\R^{(n-1)\times(n-1)}$ (resulting
  from fixing $\beta_{i,j}$ for all $j<n$), with the following property:
  for a suitable right-hand side $q$, $\LCP(M,q)$ with  
  \[M = \left(\begin{array}{cc}
      M' & b \\
      0 & 1
      \end{array}\right)
    \]
  yields $2^{\Omega(n^2)}$ different USOs in the subcube $F$
  corresponding to vertices with $v_n=1$, when $b$ is varied. This
  number is the term for $m=n$ in (\ref{eq:bigprod}).

  Since the subcube $F$ corresponds to the solutions of $w-Mz=q$ that
  satisfy $w_n=0$, we have $z_n=q_n$ within $F$. With
  $w'=(w_1,\ldots,w_{n-1})^T$, $z'=(z_1,\ldots,z_{n-1})^T$ and
  $q'=(q_1,\ldots,q_{n-1})^T$, it follows that
  \[w-Mz=q, \quad w^Tz = 0, \quad w_n=0\]
  if and only if
  \[w'-M'z'= q'+bq_n, \quad w'^Tz' = 0, \quad z_n = q_n.\]
  
  This is easily seen to imply that the induced USO in the subcube $F$
  is generated by $\LCP(M',q'+bq_n)$. Thus, $u(M')=2^{\Omega(n^2)}$, and the
  theorem is proved. 
\end{proof}

\section{Locally uniform USOs and K-matrices}
\label{sec:loc}

Finally we present a note on the relationship between K-matrices and
locally uniform USOs.

\begin{theorem}
Let $M$ be a P-matrix. $M$ is a $K$-matrix if and only if for all
nondegenerate $q$, the USO induced by $\LCP(M,q)$ is locally uniform. 
\end{theorem}

\begin{proof}
The ``only-if'' direction is Proposition~5.3
in~\cite{FonFukGar:Pivoting}. For the if-direction, suppose that $M$ is
not a K-matrix. We will construct a vector $q$ such that the induced
USO violates~\eqref{eq:up-uni}. First, since $M$ is not a K-matrix,
there exists an off-diagonal entry $m_{ij}>0$, $i\neq j$. W.l.o.g.\
assume that $\{i,j\}=\{1,2\}$ and define
\[Q = \left(\begin{array}{rr}
m_{11} & m_{12} \\
m_{21} & m_{22}
\end{array}\right).
\]

Let us now consider $B=\{1,2\}$. Then 
\[A_B = \left(\begin{array}{c|c}
-Q & 0 \\ \hline
0 & I_{n-2} 
\end{array}\right),\]
and
\[A_B^{-1} = 
\left(\begin{array}{c|c}
-Q^{-1} & 0 \\ \hline
0 & I_{n-2} 
\end{array}\right),
\]
where
\[-Q^{-1} = \frac{1}{\det(Q)}\left(\begin{array}{rr}
  -m_{22} & m_{12} \\
  m_{21} & -m_{11}
  \end{array}\right) = \frac{1}{m_{11}m_{22}-m_{21}m_{12}}\left(\begin{array}{rr}
  -m_{22} & m_{12} \\
  m_{21} & -m_{11}
  \end{array}\right).
\]
Since $Q$ is a P-matrix, its determinant is positive, hence $-Q^{-1}$ has
some positive off-diagonal entry. Suppose first that $m_{12}>0$. Then we set
$q=(-m_{12}, -(m_{22}+1),0,\ldots,0)$ and observe that
\[(A_{B}^{-1}q)_1 = \frac{-m_{12}}{m_{11}m_{22}-m_{21}m_{12}} < 0.\]
Slightly perturbing $q$ such that it becomes nondegenerate will not change
this strict inequality. But this is a contradiction to~\eqref{eq:up-uni}:
at $B=\emptyset$, the edges in directions $1$ and~$2$ are outgoing due to
$q_1,q_2<0$ (note that $m_{22}>0$ because $M$~is a P-matrix), but
at $B=\{1,2\}$, the edge in direction $1$ is \emph{not}
incoming as required by~\eqref{eq:up-uni}. If $m_{21}>0$, the vector
$q=(-(m_{11}+1),-m_{21},0,\ldots,0)$ leads to the same contradiction.
\end{proof}

\begin{remark}
It may be more interesting to answer the following open question: Is it
true that every locally uniform P-USO is a K-USO?
\end{remark}

\section*{Acknowledgments}

The third author is supported by the project
	`A Fresh Look at the Complexity of Pivoting in Linear Complementarity'
	no.~200021-124752$\,$/$\,$1 of the Swiss National Science Foundation.

\end{document}